\documentclass[11pt,reqno,a4paper,final]{amsart}
\pdfoutput=1 %% this line is needed for arxiv to use pdflatex

\usepackage{latexsym,amsfonts,amssymb,graphicx,amsmath,bm,amsfonts}
\usepackage{epstopdf}
\usepackage{amssymb}
\usepackage{enumerate}
\usepackage{enumitem}
\usepackage{booktabs}
\usepackage{stmaryrd}
\usepackage{prelim2e}
\usepackage{mathtools}
\usepackage{lineno}  

\usepackage{xcolor}
\definecolor{rltred}{rgb}{0.75,0,0}
\definecolor{rltgreen}{rgb}{0,0.5,0}
\definecolor{rltblue}{rgb}{0,0,0.75}
\definecolor{darkcerulean}{rgb}{0.03, 0.27, 0.49}

\usepackage[colorlinks=true,
urlcolor=rltblue,       % \href{...}{...} external (URL)
filecolor=rltgreen,     % \href{...} local file
citecolor=darkcerulean,     % \cite{...} 
linkcolor=darkcerulean,       % \ref{...} and \pageref{...}
]{hyperref}

\usepackage{sidecap}
\usepackage{color,soul}
\usepackage[font=small,labelfont=bf]{caption}
\numberwithin{equation}{section}

\usepackage{tikz}
\usetikzlibrary{arrows,shapes}
\usetikzlibrary{arrows, decorations.markings,fit}
\usetikzlibrary{calc,3d}
\usetikzlibrary{arrows.meta}
\allowdisplaybreaks[3]

\newcommand{\remove}[1]{}

\newcommand{\Forall}{\;\forall\,}

\newcommand{\Eh}{\mathcal{E}_h}
\newcommand{\Sh}{\mathcal{S}_h}
\newcommand{\calS}{\mathcal{S}}
\newcommand{\calT}{\mathcal{T}}
\newcommand{\calC}{\mathcal{C}}

\newcommand{\dive}{\operatorname{div}}

\newcommand{\support}{\operatorname{supp}}

\newcommand{\bv}{{\bm v}}
\newcommand{\bn}{{\bm n}}
\newcommand{\bw}{{\bm w}}
\newcommand{\bH}{{\bm H}}
\newcommand{\bV}{{\bm V}}
\newcommand{\bP}{{\bm P}}
\newcommand{\bPi}{{\bm \Pi}}
\newcommand{\bpsi}{{\bm \psi}}

\newcommand{\Th}{\mathcal{T}_h}

\newcommand{\dx}{\,\mathrm{d}x}
\newcommand{\ds}{\,\mathrm{d}s}

\providecommand{\normtmp}[2]{{#1\lVert #2 #1\rVert}}
\providecommand{\norm}[1]{\normtmp{}{#1}}

\providecommand{\abstmp}[2]{{#1\lvert #2 #1\rvert}}
\providecommand{\abs}[1]{\abstmp{}{#1}}

\newtheorem{theorem}{Theorem}
\newtheorem{lemma}[theorem]{Lemma}
\newtheorem{proposition}[theorem]{Proposition}

\theoremstyle{definition}

\newtheorem{remark}[theorem]{Remark}
\newtheorem{definition}[theorem]{Definition}
\makeatletter
\@namedef{subjclassname@2020}{%
	\textup{2020} Mathematics Subject Classification}
\makeatother
	\title[A local Fortin projection for the Scott--Vogelius elements]{A local Fortin projection for the Scott--Vogelius elements on general meshes}
	
    \author[F. Eickmann ]{Franziska Eickmann\textsuperscript{\textasteriskcentered}}
	\address{\textsuperscript{\textasteriskcentered}Department of Mathematics, Technische Universität Darmstadt, Dolivostrasse 15, 64293 Darmstadt, Germany}
	\email{eickmann@mathematik.tu-darmstadt.de}
    \email{tscherpel@mathematik.tu-darmstadt.de}

	\thanks{}
    
	\author[J. Guzm\'an]{Johnny Guzm\'an\textsuperscript{\textdagger}}
	\address{\textsuperscript{\textdagger} Division of Applied Mathematics, Brown University, Providence, RI 02912, USA}
	\email{johnny\_guzman@brown.edu}
	\thanks{}

	\author[M. Neilan]{Michael Neilan\textsuperscript{\textdaggerdbl}}
	\address{\textsuperscript{\textdaggerdbl}Departments of Mathematics, University of Pittsburgh}
	\email{neilan@pitt.edu}
	\thanks{}
    
	\author[R. Scott]{L. Ridgway Scott\textsuperscript{\textsection}}
	\address{\textsuperscript{\textsection}Departments of  Computer Science and Mathematics,
		Committee on Computational and Applied Mathematics,
		University of Chicago, Chicago IL 60637, USA}
	\email{ridg@uchicago.edu}
	\thanks{}
	
	\author[T. Tscherpel]{Tabea Tscherpel\textsuperscript{\textasteriskcentered}}
	\thanks{}

\begin{document}

%\linenumbers

	\begin{abstract}  
		We construct a local Fortin projection for the Scott--Vogelius finite element pair for polynomial degree $k \ge 4$ on general shape-regular triangulations in two dimensions. 
        In particular, the triangulation may contain singular vertices.
        In addition to preserving the divergence in the dual of the pressure space, the projection preserves discrete boundary data and satisfies local stability estimates.
	\end{abstract}

	\maketitle
	
	\medskip
	
	%   ------- KEYWORDS   ------  %
	\smallskip
	
	%   ------   SUBJECT CLASS   ------   %
	\date{}

	% ---------------------------------------------- %
	\section{Introduction}
	% ---------------------------------------------- %

            \setlength\itemsep{.3em}

	The Scott--Vogelius \cite{scott1985norm} finite elements (FE) are the first pair of inf-sup stable FE spaces that produce divergence-free velocity approximations for incompressible flow. 
	On a general two-dimensional simplicial triangulation, the velocity space is composed of globally continuous piecewise polynomials of degree $k$ ($k \ge 4$), whereas the pressure space consists
	of piecewise polynomials of degree $\le k-1$ with certain constraints at so-called singular vertices. 
	
	Fortin projections are a common tool for the analysis of FE spaces for fluid flow problems \cite{brezzi2008mixed}. 
	It is well-known that inf-sup stability implies the existence of a Fortin operator; however, this operator
	is not necessarily local and does not in general have trace-preserving properties. 
    
    Here, we construct a local Fortin projection for the Scott--Vogelius element on general triangulations, including ones with singular vertices. 
    Local Fortin projections play an essential role in several areas
    of FE analysis.
    For example, they are a critical component in the error analysis in $L^\infty$-norms  
    for incompressible Stokes and Navier--Stokes flow \cite{girault2015max,  guzman2012pointwise}. 
	Furthermore, in the numerical analysis of nonlinear fluid equations arising in non-Newtonian fluid flow, they are a key tool; see, e.g.,~\cite{BelenkiBerselliDieningEtAl2012,DKS.2013,ST.2020}. 
	In addition, Fortin operators that allow for inhomogeneous traces are central for the numerical analysis of fluid problems  with inhomogeneous Dirichlet boundary conditions~\cite{EickmannScottTscherpel2025,JessbergerKaltenbach2024}. 

    A local Fortin operator has recently been constructed in~\cite{ParkerSueli2025} for two-dimensional triangulations without singular vertices. Therein, the dependence of the stability constants on the polynomial degree is the main focus and challenge.    
	In contrast, here we treat general triangulations, including 
    those with singular vertices.
    However, we do not address the dependence on the polynomial degree in this work. 
    For velocity spaces with zero trace conditions, a local Fortin operator for general meshes can be found in \cite[A.2.2.2]{Tscherpel2018}, building on the inf-sup stability arguments in \cite{GuzmanScott19}. 
    Here, we treat the general case of a local operator that 
    applies also for inhomogeneous boundary conditions.

While singular (and nearly singular) vertices can be eliminated, e.g., 
by using a barycenter refinement, such procedures are not compatible with
commonly used adaptive mesh refinement strategies such as newest vertex bisection~\cite{Mitchel1991}.
Consequently, singular vertices are not merely a technical inconvenience but arise naturally within standard refinement routines.
Boundary singular vertices can generally be avoided in such an adaptive routine, for instance by enforcing full bisection of boundary elements, and are therefore less critical.
In the context of inhomogeneous Dirichlet boundary conditions for Scott–Vogelius elements, the exclusion of boundary singular vertices is sometimes assumed, see \cite{EickmannScottTscherpel2025}, though this restriction is not particularly severe. In contrast, interior singular vertices are typically generated during successive bisection steps and thus cannot be systematically avoided, making their treatment a practically relevant issue.
	
	Our Fortin projection is a correction of the discrete trace-preserving version of the Scott--Zhang interpolant~\cite{ScottZhang90}. 
    The corresponding correction operator, denoted by $\Pi^2$ below, depends on a key lemma (see Lemma \ref{mainlemma}), whose proof relies on the results in \cite{GuzmanScott19}. 
	
	In Section \ref{preliminaries} we give the definition of the Scott--Vogelius elements along with some preliminary results. In Section \ref{fortinsection} we construct the Fortin projection and we prove stability estimates. 
    Finally, in the appendix we prove Lemma~\ref{mainlemma} and Lemma~\ref{mainlemmaWithSlip}.

	% ---------------------------------------------- %
	\section{Preliminaries}\label{preliminaries}
	% ---------------------------------------------- %

	We assume that $\Omega \subset \mathbb{R}^2$ is a bounded polygonal domain. We let $\{ \mathcal{T}_h\}_h$ be a family of conforming, nondegenerate (shape regular) triangulations of $\Omega$ decomposing $\Omega$ into a set of closed triangles; see \cite{brenner2007mathematical}.   
	We set the local mesh size as $h_T \coloneqq  {\rm diam}(T)$ for all $T\in \calT_h$ and the maximal mesh size is denoted by $h \coloneqq \max_{T\in \calT_h} h_T$.
	The sets of vertices and  edges
     of $\mathcal{T}_h$ are denoted by
	\begin{alignat*}{1}
		\Sh=&\{ z\colon  z  \text{ is a vertex of } \Th \},\qquad 
		\Eh=\{ e\colon e \text{ is an edge of }  \Th\},
	\end{alignat*}
	respectively. 
	We also define the sets of edges and triangles that have $z \in \Sh$ as a vertex, i.e., 
	\begin{alignat*}{1}
		\Eh(z)=&\{ e \in \Eh\colon z \text{ is a vertex of } e \},\qquad
		\Th(z)=\{ T \in \Th\colon z \text{ is a vertex of } T \}.  
	\end{alignat*}
	Finally, we define the patch of a vertex $z \in \Sh$ by
	\begin{equation*}
		\Omega_h(z) \coloneqq \bigcup_{T \in \Th(z)} T,
	\end{equation*}
	with diameter denoted by $h_z\coloneqq \text{diam}(\Omega_h(z))$.

	To define the pressure space of the finite element pair we first need to define \emph{singular} and \emph{non-singular} vertices. 
	Let $z \in \Sh$ and label the set of triangles abutting $z$ as
	$\Th(z)=\{ T_1, T_2, \ldots T_L \}$ for some $L \in \mathbb{N}$. 
    If $z$ is a boundary vertex then we enumerate the triangles such that $T_1$ and  $T_L$ have a boundary edge.  Moreover, we enumerate them so that $T_j, T_{j+1}$ share an edge for $j=1, \ldots L-1$ and  $T_L$ and $T_1$ share an edge in case $z$ is an interior vertex. Let $\theta_j $ denote the angle between the edges of $T_j$ originating from $z$.   We define 
	\begin{align*}
		\Theta(z) \coloneqq
		\begin{cases}
			\max \{ |\sin(\theta_1+\theta_{2})|,  \ldots, |\sin(\theta_{L-1}+\theta_{L})|, |\sin(\theta_L+\theta_1)| \}& \text{ if }  z \text{ is an interior vertex, } \\
			\max \{ |\sin(\theta_1+\theta_{2})|,  \ldots, |\sin(\theta_{L-1}+\theta_{L})| \} & \text{ if }  z \text{ is a boundary vertex}.
		\end{cases}
	\end{align*}
	
	%%%%%%%%%%%%%%%%%%%%%%%%%%%
	%%%%%%%%%%%%%%%%%%%%%%%%%%%
	\begin{definition}
		A vertex $z \in \Sh$ is a \emph{singular vertex} if $\Theta(z)=0$. It is \emph{non-singular} if $\Theta(z) >0$.
	\end{definition}

	We denote the set of non-singular vertices by
	\begin{equation*}
		\Sh^1 \coloneqq \{ z \in \Sh\colon  z \text{ is non-singular} \},
	\end{equation*}
	and all singular vertices by $\Sh^2 \coloneqq\Sh \backslash \Sh^1$. Finally, $\mathring{\calS}_h^2$ denotes the set of interior singular vertices, and $\Sh^{2, \partial}$ denotes the set of boundary singular vertices.   
	Clearly, we have $\calS_h = \calS_h^1\cup \calS_h^2$ and $\calS_h^2 = \mathring{\calS}_h^2 \cup \calS^{2,\partial}_h$.

	We will use the following proposition due to \cite{Sauter23}. 
    
	%%%%%%%%%%%%%%%%%%%%%%%%
	%%%%%%%%%%%%%%%%%%%%%%%%
	%%%%%%%%%%%%%%%%%%%%%%%%
	\begin{proposition}[{\cite[Lem.~2.10]{Sauter23}}]\label{prop:Sauter}
		If $z \in \mathring{\calS}_h^2$ then $\Omega_h(z)$ does not intersect any other $y \in \Sh^2$.   
	\end{proposition}
	However, boundary singular vertices might share an edge with other boundary singular vertices. 
    Therefore, we decompose the set of boundary singular vertices as follows:
	\begin{equation*}
		\Sh^{2,\partial}= \bigcup_{j=1}^N B_h^j,
	\end{equation*}
	for some $N \in \mathbb{N}$, where for each $j$, $B_h^j$ is a collection of boundary singular vertices that are connected by a boundary edge path. Furthermore, each $B_h^j$ is assumed to be maximal, meaning that if $z$ is boundary singular vertex  such that $z$ is connected to a $y \in B_h^j$ via a boundary edge path, then $z \in B_h^j$.

	For each collection $B_h^j$, for $j \in \{1, \ldots, N\}$, we denote its patch as follows:
	\begin{equation}\label{def:sing-bd-patch}
		M_h^j\coloneqq \bigcup_{z \in B_h^j} \Omega_h(z).
	\end{equation}
	We also set 
	\begin{equation*}
		\Th^1\coloneqq\{ T \in \Th\colon \text{ all the vertices of } T \text{ are non-singular} \}.    	\end{equation*}

	We consider  the collection 
	\begin{equation}\label{eq:Ch-collection}
		\mathcal{C}_h\coloneqq \{ T\colon  T \in \Th^1\} 
		\cup 
		\{ \Omega_h(z)\colon    z \in \mathring{\calS}_h^2 \} 
		\cup  
		\{ M_h^j\colon  1 \le j \le N\}  ,   
	\end{equation}
	which forms a decomposition of the domain (cf.~Figure \ref{fig:Triangulation}) in the sense that 
	\begin{equation}\label{eqn:decomposition}
		\overline \Omega= \bigcup_{D \in \mathcal{C}_h} D. 
	\end{equation}
    From Proposition~\ref{prop:Sauter} and since each $B_h^j$ is maximal,
	we see that the interior of these sets in $\mathcal{C}_h$ are disjoint, i.e.,
	\eqref{eqn:decomposition} is a non-overlapping decomposition.

\begin{figure}
\begin{tikzpicture}[scale=3,rotate=90]
  \coordinate (A) at (0,2);
  \coordinate (B) at (1,2);
  \coordinate (C) at (2,2);
  \coordinate (D) at (2,1);
  \coordinate (G) at (2,-1);
  \coordinate (H) at (0,-1);
  \coordinate (F) at (0,1);
  \coordinate (E) at (1,1);
  \coordinate (I) at (2/3,1/3);
  \coordinate (J) at ($(I)!0.5!(H)$);

  \fill[blue!20] (E) -- (A) -- (F) -- cycle; % EAF
  \fill[blue!20] (A) -- (E) -- (B) -- cycle; % AEB
  \fill[blue!20] (G) -- (D) -- (E) -- cycle; % GDE

  \fill[red!20] (D) -- (C) -- (B) -- (E) -- cycle;

  \fill[green!20] (G) -- (J) -- (H) -- cycle; % GJH
  \fill[green!20] (G) -- (I) -- (J) -- cycle; % GIJ
  \fill[green!20] (G) -- (E) -- (I) -- cycle; % GEI
  \fill[green!20] (E) -- (F) -- (I) -- cycle; % EFI

  \draw[thick]
    (A) -- (B) -- (C) -- (D) -- (G) -- (H)
    (A) -- (F);

  \draw
    (A) -- (E)
    (B) -- (D)
    (C) -- (E)
    (E) -- (D)
    (E) -- (G)
    (E) -- (I) -- (H)
    (F) -- (I) -- (G)
    (E) -- (B)
    (J) -- (G)
    (I) -- (G);

  \draw[thick]
    (F) -- (I) -- (J) -- (H);

\draw (E)--(F);
\end{tikzpicture}
\caption{\label{fig:Triangulation}A pictorial description of the decomposition of the domain $\calC_h$. 
The components 
of $\calT_h^1$, $\{\Omega_h(z):z\in \mathring{\mathcal{\calS}}_h^2\}$, 
and $\{M_h^j:1\le j\le N\}$ are depicted in blue, red, and green, respectively. Note that the cardinality of $\calC_h$ in this example is five.}
\end{figure}
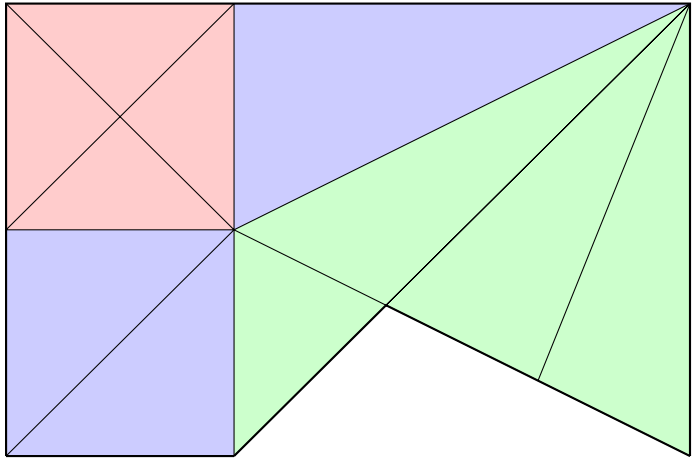
    For a domain $M\subset \Omega$ with
    $M = \cup_{T\in \mathcal{M}_h} T$ for some sub-triangulation $\mathcal{M}_h\subset \calT_h$, we set 
	\begin{equation}\label{def:ThetaD}
		\Theta(M)\coloneqq \mathop{\min_{z \in \calS_h^1}}_{z \text{ a vertex in } \mathcal{M}_h} \Theta(z)>0. 
	\end{equation}
	as the measure of ``near singularity'' with respect to $M$.
	We also define the patch of $M$, given by
	\begin{equation}
		P(M)\coloneqq\bigcup_{z \text{ vertex in } \mathcal{M}_h} \Omega_h(z).
	\end{equation}
    With slight abuse of notation we also write $P(\mathcal{M}_h)$ as the patch for a  sub-triangulation $\mathcal{M}_h$. 
	We also denote $h_D \coloneqq \operatorname{diam}(D)$ for any $D \in \mathcal{C}_h$.
    \begin{remark}\label{rem:hDsize}        
  Note that by the shape regularity of the triangulation and by Proposition~\ref{prop:Sauter},
    for $h$ sufficiently small, the number of triangles contained in $D\in \mathcal{C}_h$ is uniformly bounded.  Thus, the diameter $h_D$ is proportional to $\min_{T \subset D} h_T$ with constant only depending on the shape regularity constant.
    \end{remark}
	
    In the following $C$, denotes a generic constant that may change from line to line and only depends on the shape regularity constant, the polynomial degree $k \in\mathbb{N}$ introduced below and in some cases also the Lebesgue exponent $p \in [1,\infty]$. 
	
	% ---------------------------------------------- %
	\subsection*{Finite Element Spaces}
	% ---------------------------------------------- %
	
	Let $q$ be a piecewise continuous function, i.e., $q|_T \in C^0({T})$ for all $T \in \Th$. 
    For each singular vertex $z \in \Sh^2$ we define
	\begin{equation}\label{eq:altern-sum}
		A_h^z(q) \coloneqq \sum_{j=1}^L (-1)^{L-j} q|_{T_j}(z),
	\end{equation}
	where $L$ is the cardinality of $\calT_h(z)$. 
    Then for $k \in \mathbb{N}$, the Scott--Vogelius finite element pair of spaces is given by:
	\begin{alignat*}{1}
		\bV_h^k=&\{\bv \in \bH^1(\Omega)\colon \bv|_T \in \bP^k(T) \,\Forall T \in \Th\}, \\
		Q_h^{k-1}=& \{ q \in {\mathring{L}^2(\Omega)}\colon q|_T \in P^{k-1}(T) \Forall T \in \Th,\; A_h^z(q)=0 \,
		\Forall  z  \in \Sh^2 \}.
	\end{alignat*}
   The alternating sum condition \eqref{eq:altern-sum} reduces the dimension of the pressure space by one per singular vertex compared to the DG space of polynomial degree $k-1$. It is key to ensure that the divergence of the velocity space is the pressure space, see \cite[Lem.~2]{GuzmanScott19}.
    
	Here, $P^k(T)$ is the space of polynomials of degree less than or equal to $k$ defined on $T$,
	$\bP^k(T) = [P^k(T)]^2$, and $\bH^1(\Omega) = [H^1(\Omega)]^2$. 
	Furthermore, for $p \in [1,\infty]$, by $W^{1,p}(\Omega)$  we denote the standard Sobolev spaces, as usual we set $H^1(\Omega) \coloneqq W^{1,2}(\Omega)$. 
    By $L^p(\Omega)$ we denote the standard Lebesgue spaces, and ${\mathring{L}^2(\Omega)}$ denotes the Lebesgue function space with zero mean integral. 
	We also define the discrete velocity space 
    with zero trace as 
	\begin{align}
		\mathring{\bV}_h^k =\bV_h^k \cap \mathring{\bH}^1(\Omega),
	\end{align}
	where $\mathring{\bH}^1(\Omega)$ denotes the space of vector-valued Sobolev functions with zero traces. 
    Note that 
    \begin{align}\label{eq:divV}
        \dive \mathring{\bV}_h^k\subset Q_{h}^{k-1} \quad \text{ for } k \geq 1,
    \end{align} and it has been shown that $\dive \mathring{\bV}_h^k= Q_{h}^{k-1}$ for $k \geq 4$, see \cite{GuzmanScott19,scott1985norm}.
	
	We denote the subspace of $Q_h^{k-1}$, consisting of functions with vanishing mean
	on each triangle of $\calT_h$,  by
	\begin{equation}\label{def:Qhperp}
		Q_h^{k-1, \perp}\coloneqq
		\Big\{ q \in Q_h^{k-1}\colon \int_T q \dx =0 \quad  \forall T \in \Th \Big\},
	\end{equation}
	and also set
	\begin{equation}
		Q_h^{k-1, \perp}(\mathcal{C}_h) \coloneqq \Big\{ q \in Q_h^{k-1}\colon  \int_D q \dx =0 \quad  \forall D \in \mathcal{C}_h \Big\}. 
	\end{equation}
	For  $D \in \mathcal{C}_h$ as in~\eqref{eq:Ch-collection}, we consider the sets of pressure functions supported on the patch $D$ by
	\begin{equation*}
		Q_h^{k-1}(D)\coloneqq \{  q \in Q_h^{k-1} \colon \support(q)\subset  D \},\qquad
		Q_h^{k-1, \perp}(D)\coloneqq Q_h^{k-1,\perp}\cap Q_h^{k-1}(D).
	\end{equation*}
	Notice that $\int_D q\dx=\int_{\Omega} q\dx=0$ for all $q \in Q_h^{k-1}(D)$ since  $q$ has support in $D$ and has vanishing mean.  Hence, we have that   \begin{align}\label{Qsubset}
		Q_h^{k-1,\perp}(D)\subseteq Q_h^{k-1}(D) 
  \subseteq  Q_h^{k-1, \perp}(\mathcal{C}_h) \qquad \text{ for all } D \in \mathcal{C}_h.
	\end{align}

	Using the results in~\cite{GuzmanScott19}, we have the following
	statement on the local surjectivity properties
	of the divergence operator acting on $\mathring{\bV}_h^k$.
	For the sake of completeness, we provide a proof of the following lemma in 
	Appendix~\ref{appendix}. 
    
	%%%%%%%%%%%%%%%%%%%%%%%%%%%%%%
	%%%%%%%%%%%%%%%%%%%%%%%%%%%%%%
	%%%%%%%%%%%%%%%%%%%%%%%%%%%%%%

		\begin{lemma}\label{mainlemma}
			Let $D\in \calC_h$, $k\ge 4$, and $p\in [1,\infty]$.
			Then for any $q \in Q_h^{k-1}(D)$, there exists $\bv\in \mathring{\bV}_h^k$ such that
			\begin{enumerate}[label = (\roman*)]
            		\setlength\itemsep{.3em}
				\item $\dive \bv=q$, 
				\item  the support of $\bv$ is contained in $P(D)$, and
				\item $\|\nabla \bv\|_{L^p(P(D))} \le C_D \| q\|_{L^p(D)}$,\medskip\\
				with $C_D = C\left(1+ \frac{1}{\Theta(D)}\right)$, for $\Theta(D)$ as in \eqref{def:ThetaD}  and $C>0$ depending only on $k$, $p$, and on the shape
				regularity of the triangulation $\calT_h$ restricted to $P(D)$. 
			\end{enumerate}
		\end{lemma}

    Lemma~\ref{mainlemma} ensures the existence of a local discrete Bogovski\u{\i} operator, i.e., a right-inverse 
    of the divergence operator, cf.~\cite{BelenkiBerselliDieningEtAl2012}. 
    Other than the corresponding `continuous' local operator, which is a bounded mapping from $\mathring{L}^p(D) \to \mathring{\bm W}^{1,p}(D)$ only for $p \in (1,\infty)$ \cite[Thm.~5.2]{DieningRocircuzickaSchumacher2010}, stability of the discrete operator holds for all $p \in [1,\infty]$. 

	Setting
	\begin{equation}\label{eq:Vhk-D}
		\mathring{\bV}_h^k(D)\coloneqq \{ \bv \in \mathring{\bV}_h^{k}\colon  \dive \bv \in Q_h^{k-1}(D), \ {\rm supp}(\bv) \subset P(D)\},   
	\end{equation}
	 we have by definition that $	\dive \mathring{\bV}_h^k(D) \subset  Q_h^{k-1}(D)$, and Lemma~\ref{mainlemma} shows that
	\begin{equation}\label{eqn:range}
		\dive \mathring{\bV}_h^k(D)=  Q_h^{k-1}(D).  
	\end{equation}
	We also consider the divergence-free subspace of $\mathring{\bV}_h^k(D)$
	\begin{equation}
		\mathring{\bV}_h^{k, 0} (D)\coloneqq \{\bv \in \mathring{\bV}_h^k(D)\colon \dive \bv=0 \},
	\end{equation}
	so that by the rank-nullity theorem and \eqref{eqn:range} we have
	\begin{equation}\label{eq:dim}
		\dim \mathring{\bV}_h^{k} (D)= \dim  Q_h^{k-1}(D) +\dim \mathring{\bV}_h^{k, 0} (D). 
	\end{equation}

	% ---------------------------------------------- %
	\section{The Fortin Projection}\label{fortinsection}
	% ---------------------------------------------- %
    As is customary, we construct the Fortin projection by correcting a suitable quasi-interpolation operator ${\bm \Pi}^1 \colon {\bm W}^{1,1}(\Omega) \to {\bm V}^k_h$ by %using 
    a correction operator ${\bm \Pi}^2 \colon {\bm W}^{1,1}(\Omega)\to {\bm V}^k_h$.
    Specifically, the Fortin projection is given by
	\begin{equation}\label{eqn:FortinDef-0}
		\bPi \bv=\bPi^1 \bv+ \bPi^2(\bv-\bPi^1\bv) \qquad \text{ for } \bv \in {\bm W}^{1,1}(\Omega). 
	\end{equation}
	The correction operator will be the sum of the local operators $\bPi_D^2\colon {\bm W}^{1,1}(\Omega)  \rightarrow \mathring{\bV}_h^k(D)$, for $D \in \mathcal{C}_h$ (cf.~\eqref{eq:Ch-collection}) satisfying 
	
	\begin{subequations}\label{Pi2}
		\begin{alignat}{2}
			\int_{\Omega}  \dive \bPi_D^2 \bv \,  q \dx = & \int_{\Omega} \dive \bv\, q \dx \quad &&\text{ for any } q \in  Q_h^{k-1}(D), \label{Pi2_1}\\
			\int_{\Omega} \bPi_D^2 \bv \cdot \bw \dx =& \int_{\Omega }  \bv \cdot \bw \dx \qquad  &&\text{ for any } \bw  \in  \mathring{\bV}_h^{k, 0} (D). \label{Pi2_2}
		\end{alignat}
	\end{subequations}
    
	%%%%%%%%%%%%%%%%%%%%%%%%%%%%%%
	%%%%%%%%%%%%%%%%%%%%%%%%%%%%%%
	%%%%%%%%%%%%%%%%%%%%%%%%%%%%%%
    \begin{lemma}\label{lem:Pi2D}
		For each $D \in \mathcal{C}_h$ the operator $\bPi_D^2 \colon W^{1,1}(\Omega) \to \mathring{\bm V}^k_h(D)$ determined by \eqref{Pi2} is well-defined and is a linear projection. 
		Moreover, the following estimate holds for $p\in [1,\infty]$:
		\begin{equation}
			\norm{\bPi_D^2 \bv}_{L^p(P(D))}  + h_D \|\nabla(\bPi_D^2 \bv)\|_{L^p(P(D))} 
            \le C \left(\|\bv\|_{L^p(P(D))}+ C_D h_D \|\dive \bv\|_{L^p(D)}\right) %\quad \text{ for any } \bv\in {\bm W}^{1,p}(\Omega), 
            \label{761}
		\end{equation}
        for any $\bv\in {\bm W}^{1,p}(\Omega)$, 
		with constant $C>0$ independent of $D  \in \mathcal{C}_h$ and of $\bv$, and with constant $C_D>0$ as in Lemma~\ref{mainlemma}. 
	\end{lemma}
	\begin{proof}
		We first show that $\bPi^2_D$ is well-defined and a projection. 
		Note that the number of degrees of freedom (dofs) given in~\eqref{Pi2} is exactly $\dim \mathring{\bV}_h^k(D)$, see \eqref{eq:dim}.  
        Suppose that the right-hand side of~\eqref{Pi2} vanishes. 
		Then by~\eqref{Pi2_1} and~\eqref{eqn:range} we have that $\dive \bPi_D^2 \bv = 0$, i.e., $\bPi_D^2 \bv \in \mathring{\bV}_h^{k,0} (D)$. Therefore, by the second set of dofs \eqref{Pi2_2}, it follows that $\bPi_D^2 \bv=0$. We conclude that the operator is well-defined
        and by similar arguments we conclude  it is a projection.
		
		To establish the stability estimate~\eqref{761}, for $p \in [1,\infty]$ let $\bv\in {\bm W}^{1,p}(\Omega)$ be arbitrary but fixed. 
        By Lemma~\ref{mainlemma} there exists ${\bm m} \in  \mathring{\bV}_h^{k} (D)$  such that $\dive {\bm m}= \dive \bPi_D^2 \bv \in Q^{k-1}_h(D)$ and 
        \begin{align}\label{est:nab-m}
\|\nabla {\bm m}\|_{L^p(P(D))} \le C_D \| \dive \bPi_D^2 \bv\|_{L^p(D)}.
        \end{align}
On the other hand,  by \eqref{eqn:range}, \eqref{Pi2_1}, and
        an inverse estimate, we have
        \[
        \|\dive \bPi_D^2 \bv\|_{L^2(D)}\le C h^{-1}_D \|\dive \bv\|_{L^1(D)}.
        \]
        Therefore, a scaling argument yields 
        \begin{equation}\begin{split}\label{est:sc-Pi2D}
        \|\dive \bPi_D^2 \bv\|_{L^p(D)}
        &\le C h_D^{2\left(\frac1p - \frac12\right)}\|\dive\bPi_D^2 \bv\|_{L^2(D)}\\
        &\le C h_D^{2\left(\frac1p-1\right)} \|\dive \bv\|_{L^1(D)}\le C \|\dive \bv\|_{L^p(D)}.
        \end{split}
        \end{equation}
		Thanks to
        $\support(\bm m) \subset P(D)$ by a Poincaré inequality and applying \eqref{est:nab-m} and \eqref{est:sc-Pi2D} we obtain
		\begin{align}\label{est:m}
			\norm{\bm m}_{L^p(P(D))}
            \leq 
            C h_D \|\nabla {\bm m}\|_{L^p(P(D))} &\le C  C_D  h_D \| \dive \bv\|_{L^p(D)},
		\end{align}
	   where we have also used the shape regularity of the triangulation. 
		
		Let ${\bm \eta}\coloneqq \bPi_D^2 \bv-{\bm m}$ and note that ${\bm \eta} \in \mathring{\bV}_h^{k,0} (D)$. 
	   By~\eqref{Pi2_2} we find  that
		\begin{equation*}
			\int_{\Omega} {\bm \eta} \cdot \bw \dx
            = \int_{\Omega }  (\bv-{\bm m}) \cdot \bw  
            \dx 
            \qquad  \text{ for any } \bw  \in  \mathring{\bV}_h^{k, 0} (D).
		\end{equation*}
		Hence, by similar reasoning as above using inverse estimates, and employing \eqref{est:m} we obtain 
		\begin{equation}\label{eqn:etaLp}
			\begin{split}
				\|{\bm \eta}\|_{L^p(P(D))} 
				&\le 
                C \| \bv-{\bm m}\|_{L^p(P(D))} 
                \le C \left(\|\bv\|_{L^p(P(D))}+ \| {\bm m}\|_{L^p(P(D))} \right)\\
                & \leq   C \left(\|\bv\|_{L^p(P(D))}+ C_D h_D \| {\dive \bv}\|_{L^p(P(D))} \right). 
			\end{split}
		\end{equation}
        Finally, by definition of $\bm \eta$ and \eqref{est:m}--\eqref{eqn:etaLp}  we obtain
		\begin{alignat*}{1}
			  	\norm{\bPi_D^2 \bv}_{L^p(P(D))} %+ h_D \norm{\nabla \bPi_D^2 \bv}_{L^p(P(D))} 
                &
                \leq C \left(\|\bv\|_{L^p(P(D))}+ C_D h_D \| {\dive \bv}\|_{L^p(P(D))} \right). 
		\end{alignat*} 
        Then, the estimates on the gradients follow by an inverse estimate, 
		which proves the claim. 
	\end{proof}

	Next we patch together the local
	projections $\bPi^2_D$, as defined in \eqref{Pi2}, to construct the operator
	$\bPi^2 \colon {\bm W}^{1,1}(\Omega) \to \mathring{\bV}_h^k$ defined by
	\begin{equation}\label{def:Pi2}
		\bPi^2 \bv \coloneqq  \sum_{D\in \mathcal{C}_h} \bPi_D^2 \bv  = \sum_{T \in \Th^1} \bPi_T^2 \bv+ \sum_{z \in \mathring{\calS}_h^2} \bPi_{\Omega_h(z)}^2 \bv +\sum_{j=1}^N \bPi_{M_h^j}^2 \bv \qquad \text{ for } \bv \in {\bm W}^{1,1}(\Omega). 
	\end{equation}
    Note that the supports of $\bPi_D \bv$, for $D \in \mathcal{C}_h$, are not pairwise disjoint. 
	To develop a local stability result for this operator, we define
$$\mathcal{C}_h(T) \coloneqq \{ D \in \mathcal{C}_h \colon T \subset P(D) \}$$ for $T \in \mathcal{T}_h$, and $$Q(T) \coloneqq \bigcup_{D \in \mathcal{C}_h(T)} P(D).$$
    
	%%%%%%%%%%%%%%%%%%%%%%%%%%%%%%
	%%%%%%%%%%%%%%%%%%%%%%%%%%%%%%
	%%%%%%%%%%%%%%%%%%%%%%%%%%%%%%
	\begin{lemma}\label{lem:GlobalPi2}
        The operator $\bPi^2 \colon {\bm W}^{1,1}(\Omega) \to \mathring{\bV}_h^k $ as defined in \eqref{def:Pi2} is linear and satisfies for all $\bv \in {\bm W}^{1,1}(\Omega)$ the following properties: 
        \begin{enumerate}[label = (\roman*)]
            \item \label{itm:Pi2-0} 
            $\bPi^2(\bm 0) = \bm 0$;
            \item  $\dive \bPi^2 \bv\in Q^{k-1,\perp}_h(\mathcal{C}_h)$ and
		\begin{equation}\label{751}
			\int_{\Omega} \dive(\bPi^2 \bv) \,  q \dx 
            =\int_{\Omega} \dive \bv \, q  \dx \qquad \text{ for any } q\in Q_h^{k-1, \perp}(\mathcal{C}_h).  
		\end{equation} 
        \item \label{itm:Pi2-stab} 
        Moreover, for $p \in [1,\infty]$ one has 
		\begin{equation}\label{eqn:Pi2Stab}
			\norm{ \bPi^2 \bv}_{L^p(T)} + h_T \|\nabla \bPi^2 \bv\|_{L^p(T)}\le C \left(  \|\bv\|_{L^p(Q(T))} +  C_{Q(T)} h_T  \|\dive \bv\|_{L^p(Q(T))} \right)
		\end{equation}
		for any $T\in \calT_h$ and any $\bv \in {\bm W}^{1,p}(\Omega)$. 
        Here, 
    \begin{align}\label{def:C-QT}
        C_{Q(T)} = 1 + \tfrac{1}{\Theta(Q(T))}
    \end{align}
    with $\Theta(Q(T))$ given in \eqref{def:ThetaD}, and the constant $C>0$ in \eqref{eqn:Pi2Stab}
    depends only on $k,p$ and on the shape regularity of $\mathcal{T}_h$ restricted to $Q(T)$. 
        \end{enumerate}		
	\end{lemma}
    
	\begin{proof}
   That $\bPi^2$ maps to $\mathring{\bV}^k_h$ follows from the fact that, by Lemma~\ref{lem:Pi2D}, $\bPi^2_D$ maps to $\mathring{\bV}^k_h(D)\subset \mathring{\bV}^k_h$, for each $D \in \mathcal{C}_h$. 
    Furthermore, we have that $\bPi^2_D(\bm 0) = \bm 0$, and thus by \eqref{def:Pi2} the same holds true for $\bPi^2$, which shows~\ref{itm:Pi2-0}. 
    
     By the definition of $\mathring{\bV}_h^k(D)$ in~\eqref{eq:Vhk-D}
		and by~\eqref{Qsubset} we have for each $\bv \in {\bm W}^{1,1}(\Omega)$ that 
		\begin{align}\label{eq:divPi2}
			\dive \bPi^2_D \bv\in Q_h^{k-1}(D)\subset Q_h^{k-1,\perp}(\mathcal{C}_h).
		\end{align} 
        Recall that the sets in $\mathcal{C}_h$ have disjoint interior, see \eqref{eq:Ch-collection}. 
		Thus, for $q\in Q_h^{k-1,\perp}(\mathcal{C}_h)$ we may consider the decomposition 
        \begin{align*}
            q= \sum_{D \in \mathcal{C}_h} q_D, \qquad \text{ where } \quad q_D \coloneqq \chi(D) q\in Q^{k-1}_h(D), 
        \end{align*} 
         and $\chi(D)$ is the characteristic function of $D \in \mathcal{C}_h$.
		Then by \eqref{def:Pi2} and \eqref{Pi2_1}, we have
        \begin{alignat}{1}
			\int_{\Omega} \dive(\bPi^2 \bv) \, q\dx 
			& = \sum_{D \in \mathcal{C}_h} \int_{\Omega}  \dive (\bPi^2 \bv) q_D \dx
			= \sum_{D \in \mathcal{C}_h} \int_{\Omega}  \dive (\bPi_D^2 \bv) q_D \dx \label{eq:Pi2div-cons}
			\\
			&= \sum_{D \in \mathcal{C}_h} \int_{\Omega}  \dive  \bv \,q_D \dx 
			= \int_{\Omega} \dive \bv \, q \dx.
		\end{alignat}
		In \eqref{eq:Pi2div-cons} we used that the support of $\dive(\bPi_D^2 \bv)$ is contained in $D$ and hence for example $\dive (\bPi_D^2 \bv) q_{D'} =0$ when $D,D' \in \mathcal{C}_h$ and $D \neq D'$.  
        This proves~\eqref{751}. 
		
		Next, to prove \eqref{eqn:Pi2Stab}, 
		we use~\eqref{def:Pi2}, Minkowski's inequality and \eqref{761} to find 
        
        \begin{alignat*}{1}
			\| \bPi^2  \bv \|_{L^p(T)}  
			& \le \sum_{D \in \mathcal{C}_h(T)}   \| \bPi_D^2  \bv \|_{L^p(T)}
             \le \sum_{D \in \mathcal{C}_h(T)}   \| \bPi_D^2  \bv \|_{L^p(P(D))}
			\\
			&\le C \sum_{D \in \mathcal{C}_h(T)}  \big( \|\bv \|_{L^p(P(D))} + C_D h_D \|\dive \bv \|_{L^p(D)}\big)\\
			&  \le C \left( 
             \|\bv \|_{L^p(Q(T))}
           +  C_{Q(T)} h_T \|\dive \bv \|_{L^p(Q(T))}\right).
		\end{alignat*}
        In the last step we used $h_T \approx h_D$ for any $D \in \mathcal{C}_h(T)$, which holds with constants depending on the shape regularity constant, cf.~Remark \ref{rem:hDsize}. 
        The estimate on the gradient follows by an inverse estimate.  
	\end{proof}

	Let $\bPi^1\colon {\bm W}^{1,1}(\Omega)  \rightarrow \bV_h^k$  (for $k \ge 2$) be a projection such that 
 for all $T\in \calT_h$, and for $p \in [1,\infty]$
	\begin{alignat}{2}
	 	\int_T \dive \bPi^1 \bv \dx &= \int_T \dive \bv \dx\qquad &&\text{for all }  \bv\in {\bm W}^{1,1}(\Omega) \label{741},\\ 
         \label{SZE-approx}
        \|  \nabla^j (\bPi^1 \bv-\bv )\|_{L^p(T)}
         & \le C h_T^{r-j} \| \nabla^r \bv\|_{L^p(P(T))}\qquad &&\text{for all } \bv\in {\bm W}^{r,p}(\Omega), 
         %\ (p\in [1,\infty]),
	 \end{alignat}
for any $j \in \{0,1\}$ and $r \in \{1, \ldots, k+1\}$, 
with constant $C$ only depending on $p,k,$ and the shape regularity of $\calT_h$. 
We further assume this operator satisfies
	\begin{equation}\label{eqn:Pi1Boundary}
		\text{if $\bv=\bv_h$ on  $\partial \Omega$ for some $\bv_h \in \bV_h^k$, then } \bPi^1 \bv=\bv  \text{ on } \partial \Omega.
	\end{equation}
	For example, $\bPi^1$ can be constructed by a modification of the Scott--Zhang interpolant using suitable degrees of freedom, see \cite[Lem. 5.4.1 and Rmk.~5.4.4]{BernardiGiraultHechtEtAl2024}. %Appendix~\ref{appendix-pi1}.
	\smallskip 
	
	Finally, the Fortin projection is defined as in \eqref{eqn:FortinDef-0}.
	% \begin{equation}\label{eqn:FortinDef}
	% 	\bPi \bv=\bPi^1 \bv+ \bPi^2(\bv-\bPi^1\bv).
	% \end{equation}

	%%%%%%%%%%%%%%%%%%%%%%%%%%%%%%
	%%%%%%%%%%%%%%%%%%%%%%%%%%%%%%
	%%%%%%%%%%%%%%%%%%%%%%%%%%%%%%
	\begin{theorem}[Fortin operator]\label{thm:Fortin-1}
		For $k\geq 4$ the operator $\bPi\colon {\bm W}^{1,1}(\Omega) \rightarrow \bV_h^k$ 
		defined by \eqref{eqn:FortinDef-0} is a linear projection with the following properties: 
		\begin{enumerate}[label = (\roman*)]
			\item \label{itm:commute} it preserves the divergence in the sense that 
			\begin{equation*}%\label{commute}
				\int_{\Omega} \dive(\bPi \bv) \, q \dx =\int_{\Omega} \dive \bv \, q  \dx \qquad \text{ for any } q\in Q_h^{k-1}
              \text{ and }\bv\in {\bm W}^{1,1}(\Omega),
			\end{equation*}
			\item \label{itm:bound} for any $p \in [1,\infty]$ there is a constant $C>0$ such that 
			\begin{equation*}%\label{bound}
				\| \nabla \bPi \bv\|_{L^p(T)} \le C(1+ C_{Q(T)}) \|\bv\|_{W^{1,p}(P(Q(T)))} \qquad \text{ for any } 
              \bv \in   {\bm W}^{1,p}(\Omega)
			\end{equation*}
          for any $  T\in \calT_h$. 
           The constant $C_{Q(T)}>0$ is as in~\eqref{def:C-QT}, and the constant $C>0$ only depends on $k,p$ and on the shape regularity of $\mathcal{T}$ restricted to $P(Q(T))$. 
            \item  \label{itm:approx} 
            for any $p \in [1,\infty]$ there is a constant $C>0$ such that 
    \begin{align*}
    \norm{\nabla^j(\bv - \bPi \bv)}_{L^p(T)} \leq C (1 + C_{P(Q(T))}) h_T^{r-j} \norm{\nabla^r \bv}_{L^p(P(Q(T)))}\quad \text{ for any }\bv\in {\bm W}^{r,p}(\Omega),
    \end{align*}
    for any $T \in \mathcal{T}_h$, $j \in \{0,1\}$, and $r \in \{1,\ldots,k+1\}$. 
    Here  with $\Theta(P(Q(T)))$ given in \eqref{def:ThetaD} we have 
    $C_{P(Q(T))} = 1 + \tfrac{1}{\Theta(P(Q(T)))}$.  
    The constant $C$ only depends on $k,p$ and on the shape regularity of $\mathcal{T}$ restricted to $P(Q(T))$.  
			\item  \label{itm:trace-pres}	if %for $\bv \in \bH^1(\Omega)$ 
            one has $\bv=\bv_h$ on  $\partial \Omega$ for some $\bv_h \in \bV_h^k$ then 	
			\begin{equation*}
				\bPi \bv=\bv  \text{ on } \partial \Omega.
			\end{equation*}
		\end{enumerate}
	\end{theorem}
	\begin{proof}
		First, to show $\bPi$ is a projection we see that if $\bv \in \bV_h^k$ then $\bv-\bPi^1 \bv=0$
		because $\bPi^1$ is a projection. Hence, by Lemma~\ref{lem:GlobalPi2} it follows that $\bPi^2(\bv-\bPi^1 \bv)=0$, and therefore $\bPi \bv=\bPi^1 \bv=\bv$.
		
		Now we prove~\ref{itm:commute}. 
        Let  $q\in Q_h^{k-1}$ and set $\bw \coloneqq \bPi^1 \bv-\bv$ for $\bv \in {\bm W}^{1,1}(\Omega)$. 
		We then have 
		\begin{equation*}
			\int_{\Omega} \dive(\bPi \bv-\bv) \, q \dx 
            =
            \int_{\Omega} \dive (\bw-\bPi^2 \bw)\, q \dx.
		\end{equation*}
		
		We decompose $q$ into $q= q_0+q_1$, where $q_0|_D =\frac{1}{|D|} \int_D q \dx$ for all $D \in \mathcal{C}_h$,
		and note that $q_1 \in Q_h^{k-1, \perp}(\mathcal{C}_h)$. 
		Moreover, by \eqref{741} we see that $\int_{\Omega} \dive \bw \,q_0\dx=0$. By Lemma~\ref{lem:GlobalPi2} we have that $\dive \bPi^2 \bw  \in Q_h^{k-1, \perp}(\mathcal{C}_h)$, and hence
		\begin{equation}
			\int_{\Omega} \dive (\bw-\bPi^2 \bw)\, q \dx 
            = \int_{\Omega} \dive (\bw-\bPi^2 \bw)\, q_1 \dx 
            =0,  
		\end{equation}
		by \eqref{751}. Therefore \ref{itm:commute} is satisfied.

		To derive the stability estimate \ref{itm:bound},
		we use Minkowski's inequality, \eqref{eqn:Pi2Stab}  and the stability following from \eqref{SZE-approx} with $j = r = 1$, to obtain
		\begin{alignat*}{1}
			\|\nabla \bPi  \bv\|_{L^p(T)}
			&\le \|\nabla \bPi^1  \bv\|_{L^p(T)}+ \|\nabla \bPi^2  (\bv-\bPi^1 \bv)\|_{L^p(T)} \\
			&\le  
            \|\nabla \bPi^1  \bv\|_{L^p(T)}
            + C\left(C_{Q(T)} \|\dive(\bv-\bPi^1 \bv)\|_{L^p(Q(T))} + h_T^{-1} \|\bv-\bPi^1 \bv\|_{L^p(Q(T))}\right)\\%
			&\le C (C_{Q(T)}+1) \|\nabla \bv\|_{L^{p}(P(Q(T)))}.
		\end{alignat*}
	for any $T \in \mathcal{T}_h$ and any $\bv \in {\bm W}^{1,p}(\Omega)$.   

    To prove \ref{itm:approx} we apply the stability properties of $\bPi^2$ due to Lemma~\ref{lem:GlobalPi2}~\ref{itm:Pi2-stab} for $j \in \{0,1\}$, and the approximation properties of $\bPi^1$ in \eqref{SZE-approx} and obtain 
    \begin{align*}
    h_T^j\|\nabla^j(\bv - \bPi \bv)\|_{L^p(T)} 
     &\leq 
        h_T^j\norm{\nabla^j(\bv - \bPi^1 \bv)}_{L^p(T)}  + 
            h_T^j\norm{\nabla^j \bPi^2 (\bv - \bPi^1 \bv)}_{L^p(T)} \\
    & \leq 
    C  \left( 
    \norm{\bv - \bPi^1 \bv}_{L^p(Q(T))} + (1 + C_{Q(T)}) h_T 
    \norm{\nabla (\bv - \bPi^1 \bv)}_{L^p(Q(T))} 
   \right) \\
    &  \leq C(1 + C_{P(Q(T))}) h_T^r \norm{\nabla^{r} \bv}_{L^p(P(Q(T)))},
    \end{align*}
    where we have used shape regularity of $\mathcal{T}_h$. 
    
	Finally, if $\bv = \bv_h$ on $\partial\Omega$ for some $\bv_h\in \bV_h^k$, then $\bPi^1 \bv = \bv$ on $\partial\Omega$ by \eqref{eqn:Pi1Boundary}. 
	Since $\bPi^2$ maps to $\mathring{\bV}^k_h$, we have that $\bPi^2 (\bv - \bPi^1 \bv)\in \mathring{\bV}_h^k$ vanishes
	on $\partial\Omega$. This allows us to conclude that $\bPi \bv = \bv$ on $\partial\Omega$ by the definition of $\bPi$, which proves~\ref{itm:trace-pres}. 
		
	\end{proof}

    \begin{remark}[trace preservation properties]\label{rmk:traces-1}
    For problems with inhomogeneous Dirichlet boundary conditions it can be useful to have further boundary preservation properties, cf.~\cite{EickmannScottTscherpel2025,EickmannTscherpel2025}. 
    By construction,  $\bPi$ has the same boundary preservation properties as $\bPi^1$, since $\bPi^2$ maps to $\mathring{\bV}^k_h$. 
    For example, $\bPi^1$ can be chosen such that 
    \begin{enumerate}[label= (\roman*)]
        \item \label{itm:trace-1} (zero normal trace) $\bPi^1 \bv \cdot \bn = 0$  on $\partial \Omega$   for any $ \bv \in {\bm W}^{1,1}(\Omega)$  with $\bv \cdot\bn = 0$ on  $\partial \Omega$,
         where $\bn$ denotes the outer unit normal on $\partial \Omega$;
         \item \label{itm:trace-2} (mean normal trace preservation) 
         $\int_f (\bv - \bPi^1 \bv) \cdot \bn \, \mathrm{d}s(x) = 0$   for any $\bv \in {\bm W}^{1,1}(\Omega)$. 
    \end{enumerate}
  In particular, those properties transfer to the local Fortin operator $\bPi$ as defined in~\eqref{eqn:FortinDef-0} and so
  \[
  \int_\Omega \dive(\bPi \bv)\dx = \int_{\partial \Omega} (\bPi \bv)\cdot \bn\ds(x) = 
  \int_{\partial \Omega} (\bPi^1 \bv)\cdot \bn\ds(x) =
  \int_{\partial \Omega}  \bv\cdot \bn\ds(x) = \int_\Omega \dive \bv\dx
  \]
  This means that Theorem~\ref{thm:Fortin-1}~\ref{itm:commute} also holds for $q$ from the DG space of polynomial degree $k-1$ without zero mean condition. 

  Note that the original trace-preserving version of the Scott--Zhang operator in~\cite{ScottZhang90} satisfies ~\ref{itm:trace-1}, see~\cite[Lem.~4.11]{GazcaOrozcoGmeinederKokavcovaEtAl2025}. 
  Using the modification of the degrees of freedom as described in \cite[Sec. 5.4]{BernardiGiraultHechtEtAl2024} it additionally satisfies~\ref{itm:trace-2}. 
    \end{remark}
%%%%%%%%%%%%%%%%%%%%
%%%%%%%%%%%%%%%%%%%%
%%%%%%%%%%%%%%%%%%%%
    \begin{remark}\label{LemmaForMacro}
    The polynomial degree restrictions in 
    Lemma~\ref{mainlemma}, Theorem~\ref{thm:Fortin-1}
    and Remark~\ref{rmk:traces-1} can be relaxed on certain split triangulations. 
    For example,
    on Clough--Tocher triangulations, obtained 
    by subdividing each triangle into three subtriangles by connecting the vertices with an interior point,
     Lemma \ref{mainlemma} holds for $k\ge 2$ (cf.~\cite[Theorem 3.1]{GuzmanNeilan18}).
     Likewise on Powell--Sabin triangulations, which divides each triangle
     into six subtriangles, the result is true for $k\ge 1$ \cite[Theorem 1]{GuzmanEtal20}.
     In both instances, the domain $D$ is taken to be a macro element, i.e.,
     the union of three triangles in the case of Clough--Tocher triangulations,
     and the union of six triangles in the case of Powell--Sabin triangulations.
     Consequently, since the construction of $\bPi^1$ requires $k\ge 2$, 
     Theorem \ref{thm:Fortin-1} holds on these split meshes for $k\ge 2$.
     In addition, a local Fortin operator on Powell--Sabin triangulations
     in the case $k=1$ is implicitly constructed in \cite[Theorem 3.4]{FabienEtal22}
     in the case of homogeneous boundary conditions. 
     For inhomogeneous Dirichlet boundary conditions a construction is also possible by constructing $\bPi^1$ so that it satisfies the conditions in Remark~\ref{rmk:traces-1}. 
     
    \end{remark}

    \begin{remark}
        The results in \cite{BelenkiBerselliDieningEtAl2012,DKS.2013,JessbergerKaltenbach2024} are formulated under the assumption that the Fortin operator is local, with estimates
        involving only a single neighboring patch. However, the arguments in these 
        references extend to operators with larger locality, provided that the number
        of layers in the associated neighborhood remains uniformly bounded, see, e.g.,~\cite{Tscherpel2018}.
    \end{remark}
	
	%%%%%%%%%%%%%%%%%%%%%%%%%%%%%%
	%%%%%%%%%%%%%%%%%%%%%%%%%%%%%%
	%%%%%%%%%%%%%%%%%%%%%%%%%%%%%%
    \subsection*{Slip Boundary Conditions}\label{rem:Slippy}
		The above construction and proof
		can be modified to incorporate slip boundary 
		conditions instead of Dirichlet boundary conditions. This means that the velocity functions satisfy $\bv \cdot \bn = 0$ on the boundary $\partial \Omega$, where $\bn$ denotes the outer unit normal on $\partial \Omega$. 
        In this case, instead of with $\mathring{\bV}_h^k$ one works with the velocity space 
			\begin{align}
				\tilde \bV_h^k = \{\bv\in \bV_h^k\colon\ \bv\cdot \bn = 0\text{ on }\partial\Omega\}.
			\end{align}
        Due to the exact divergence constraint in the Scott--Vogelius element, the pressure space has to be adapted as well.
		In this setting, the pressure spaces are replaced by
		\begin{equation}
			\label{eqn:tQh}
			\begin{split}
				\tilde Q_h^{k-1} 
				&= \{q\in \mathring{L}^2(\Omega) \colon  q|_T\in P^{k-1}(T)\ \forall T\in \calT_h,\ A_h^z(q) = 0\ \forall z\in \mathring{S}_h^2\},\\
				\tilde Q_h^{k-1}(D)
				& = \{q\in \tilde Q_h^{k-1}\colon {\rm supp}(q)\subset D\},
			\end{split}
		\end{equation}
		that is, the weak continuity constraint is only imposed at interior singular vertices, see~\cite[Tab.~3, case I]{AinsworthParker2022}. 
        With those definitions one still has $$\dive \tilde \bV_h^k =  \tilde Q_h^{k-1}  $$ 
        and inf-sup stability holds
        for $k \geq 4$~\cite[Thm.~5.1]{AinsworthParker2022}.  
		Correspondingly, the collection of sets is modified to
		\[
		\tilde\calC_h = \{ T\colon T\in \tilde \calT_h^1\}\cup \{\Omega_h(z)\colon z\in \mathring{\mathcal{S}}_h\},
		\]
		where $\tilde \calT_h^1\subset \calT_h$ is the set of triangles 
		with no interior singular vertices.

We then have the following result, analogous to Lemma~\ref{mainlemma}
but with $Q_h^{k-1}(D)$ replaced by $\tilde Q_h^{k-1}(D)$ and $\mathring{\bV}_h^k$
			replaced by $\tilde \bV_h^k$.  Its proof is given in Appendix \ref{app:ProofWithSlip}.

%%%%%%%%%%%%%%%%%%%%%%%%%%%%%%%
%%%%%%%%%%%%%%%%%%%%%%%%%%%%%%%
%%%%%%%%%%%%%%%%%%%%%%%%%%%%%%%
\begin{lemma}\label{mainlemmaWithSlip}
Let $ D\in \tilde \calC_h$, $k\ge 4$ and $p \in [1,\infty]$.
Then for any $q \in \tilde Q_h^{k-1}(D)$, there exists $\bv\in \tilde{\bV}_h^k$ such that
\begin{enumerate}[label = (\roman*)]
\item $\dive \bv=q$, 
\item  the support of $\bv$ is contained in $P(D)$, and
\item $\|\nabla \bv\|_{L^p(P(D))} \le C_D \| q\|_{L^p(D)}$,
with $C_D = C\left(1+ \frac{1}{\Theta(D)}\right)$, with $\Theta(D)$ as in \eqref{def:ThetaD}, and $C>0$ depends only on $k$, $p$ and on the shape
regularity of the triangulation $\calT_h$ restricted to $P(D)$. 
\end{enumerate}
\end{lemma}

		The construction of an operator $\tilde \bPi^2 \colon {\bm W}^{1,1}(\Omega)\to \tilde \bV_h^k$ satisfying 
		the properties in Lemma \ref{lem:GlobalPi2} with respect
		to these modified spaces and domain decomposition then follows verbatim
        from the same arguments used to construct $\bPi^2$.		
		Finally, the assumption  \eqref{eqn:Pi1Boundary} of the operator $\bPi^1$ is
		replaced by
		\begin{equation}\label{eqn:Pi1BoundarySlip}
			\text{if $\bv\cdot\bn=\bv_h\cdot \bn$ on  $\partial \Omega$ for some $\bv_h \in \bV_h^k$, then } (\bPi^1 \bv)\cdot \bn=\bv\cdot \bn  \text{ on } \partial \Omega.
		\end{equation}
		Note that this property is satisfied for $\bPi^1$ given in  \cite[Lem. 5.4.1 and Rmk.~5.4.4]{BernardiGiraultHechtEtAl2024}.%Appendix~\ref{appendix-pi1}.

In summary, by setting $\tilde \bPi \colon {\bm W}^{1,1}(\Omega) \to \bV_h^k$ as
\begin{equation}\label{SlippyFortin}
\tilde \bPi \bv = \bPi^1 \bv+\tilde \bPi^2(\bv - \bPi^1 \bv) \qquad \text{ for } \bv \in {\bm W}^{1,1}(\Omega),
\end{equation}
we have the following result analogous to Theorem~\ref{thm:Fortin-1}.

    %%%%%%%%%%%%%%%%%%%%%%%%%%%%%%
	%%%%%%%%%%%%%%%%%%%%%%%%%%%%%%
	%%%%%%%%%%%%%%%%%%%%%%%%%%%%%%
	\begin{theorem}\label{thm:Fortin-2}
		For $k \geq 4$ the operator $\tilde \bPi\colon {\bm W}^{1,1}(\Omega) \rightarrow \bV_h^k$ 
		defined by \eqref{SlippyFortin} is a linear projection with the following properties:
		\begin{enumerate}[label = (\roman*)]
			\item  it preserves the divergence in the sense that 
			\begin{equation*}
				\int_{\Omega} \dive(\tilde \bPi \bv) \, q \dx =\int_{\Omega} \dive \bv \, q  \dx \qquad 
                \text{ for any } q\in \tilde Q_h^{k-1} \text{ and } \bv\in {\bm W}^{1,1}(\Omega),  
			\end{equation*}
			\item for any $p \in [1,\infty]$
            there is a constant $C>0$ such that 
			\begin{equation*}%\label{bound}
				\| \nabla \tilde \bPi \bv\|_{L^p(T)} \le C(1+ C_{Q(T)}) \|\bv\|_{W^{1,p}(P(Q(T)))}\qquad \text{for any } \bv \in {\bm W}^{1,p}(\Omega),
			\end{equation*}
       and for any $T\in \calT_h$.
       The constant $C_{Q(T)}>0$ is as in \eqref{def:C-QT}, and the constant $C>0$ only depends on $k,p$ and on the local shape regularity constant. 

			\item  for any $p \in [1,\infty]$ there is a constant $C>0$ such that 
    \begin{align*}
    \norm{\nabla^j(\bv - \tilde\bPi \bv)}_{L^p(T)} \leq C (1 + C_{P(Q(T))}) h_T^{r-j} \norm{\nabla^r \bv}_{L^p(P(Q(T)))}\quad\text{ for any }\bv\in {\bm W}^{r,p}(\Omega),
    \end{align*}
    for any $T \in \mathcal{T}_h$, $j \in \{0,1\}$, and $r \in \{1,\ldots,k+1\}$. 

    With $\Theta(P(Q(T)))$ as in \eqref{def:ThetaD} the constant is $C_{P(Q(T))} = 1 + \tfrac{1}{\Theta(P(Q(T)))}$.  The constant $C$ only depends on $k,p$ and on the shape regularity of $\mathcal{T}$ restricted to $P(Q(T))$. 
			\item 	if 
            one has $\bv\cdot \bn=\bv_h\cdot \bn$ on  $\partial \Omega$ for some $\bv_h \in \bV_h^k$ then 	
			\begin{equation*}
				(\tilde \bPi \bv)\cdot \bn=\bv\cdot \bn \qquad  \text{ on } \partial \Omega.
			\end{equation*}
		\end{enumerate}
	\end{theorem}

    \begin{remark}
    Similarly as in Remark~\ref{rmk:traces-1} the operator $\bPi^1$ can be chosen to satisfy the trace preservation properties~\ref{itm:trace-1}--\ref{itm:trace-2} therein, see \cite[Sec.~5.4]{BernardiGiraultHechtEtAl2024}. 
    Since $\tilde \bPi^2$ maps to $\tilde \bV_h^k$, the properties \ref{itm:trace-1}--\ref{itm:trace-2} also hold for $\tilde \bPi$ as in \eqref{SlippyFortin}. 
    \end{remark}

	\subsubsection*{Acknowledgements}
    F.E., L.R.S.\ and T.T.\ thank Charlie Parker for useful discussions.  

    The work by J.G.\ was supported in part by the National Science Foundation
    through grant DMS-2309606, and the work of M.N.\ was supported
    in part by the National Science Foundation through grant DMS-2309425. 
    The work by T.T.\ was supported by the German Research Foundation (DFG) via grant TRR 154, subproject C09, project number 239904186 and by the Graduate School CE within Computational Engineering at Technische Universität Darmstadt.

	\bibliography{lit}{}
	\bibliographystyle{amsplain}

	\appendix
	
	\section{Proof of Lemma \ref{mainlemma}}\label{appendix}
	
	\noindent We first prove the claim for $q \in Q_h^{k-1, \perp}(D)$, and then we extend it to $q \in Q_h^{k-1}(D)$ for $D \in \mathcal{C}_h$, see~\eqref{eq:Ch-collection}. 
	
	\textit{Step 1:}
	Let $q \in Q_h^{k-1, \perp}(D)$. A slight generalization of ~\cite[Lem.~6]{GuzmanScott19} shows that there exists a $\bw \in \mathring{\bV}_h^k$ {($k\ge 4$)} such that $\dive \bw=q$ on all the vertices $\Sh$,   $\text{supp } \bw \subset P(D)$,   $\dive \bw \in Q_h^{k-1, \perp}$, and 
	\begin{equation}
		\|\nabla \bw\|_{L^p(P(D))} \le C \left(\frac{1}{\Theta(D)}+1\right) \|q\|_{L^p(D)}. 
	\end{equation}
    
	Letting $r=q-\dive \bw \in Q_h^{k-1, \perp}$, we see that $r$ is supported in $P(D)$,   $r$ vanishes on all the vertices $\Sh$ and $\int_T r\dx =0$ for all $T \in \Th$. 
    Hence, by \cite[Prop.~2]{GuzmanScott19} there exists a ${\bm \eta} \in \mathring{\bV}_h^k$ which is supported in $P(D)$ such that $\dive {\bm \eta}=r$ with the bound 
	\begin{equation}
		\|\nabla {\bm \eta} \|_{L^p(P(D))} \le C  \|r\|_{L^p({P(D)})}. 
	\end{equation}
	We set $\bv= \bw+{\bm \eta}\in \mathring{\bV}_h^k$ to arrive at the result.  
	
	\textit{Step 2:} 
    Let us consider $q \in Q_h^{k-1}(D)$. 
	If $D= T$ for $T \in \Th^1$ then 
	one has $Q_h^{k-1}(T) = Q_h^{k-1,\perp}(T)$, since local and 
    global mean free condition agree for functions supported only in $T$. 
    Thus, 
   the result follows from the first step. 
	Hence, we only need to consider when $D =\Omega_h(z)$ for $z \in \mathring{\calS}_h^2$  and $D= M_h^j$ for $1 \le j \le N$. 
    Since $D$ consists of more than one simplex, the local mean-free conditions are stronger than the one global mean-free condition, i.e., $Q_h^{k-1}(D) \supsetneq Q_h^{k-1,\perp}(D)$. 
    We consider the former case of $D =\Omega_h(z)$ for $z \in \mathring{\calS}_h^2$ only as the latter is similar.

	To this end, let $q \in Q_h^{k-1}(D)$ be arbitrary but fixed. 
    The key is to find $\boldsymbol{\psi} \in \mathring{\bV}_h^k $ with  $\support(\boldsymbol{\psi}) \subset D$ such that 
	\begin{alignat}{3}\label{eq:psi-1}
		\int_{T} (q-\dive \boldsymbol{\psi} ){\dx} &=0 \quad  &&\text{ for all } T \in \Th \text{ with }  T \subset D, \\ 
		\label{eq:psi-2}
		\| \nabla \boldsymbol{\psi}  \|_{L^p(D)}  &\le   C \| q\|_{L^p(D)},&&
	\end{alignat}
    with constant only depending on $p,k$ and on the shape regularity of $\calT_h$. 
	Indeed, if this is the case, then with \eqref{eq:divV} and $\bpsi \in\mathring{\bV}_h^k $ we have $\dive \bpsi \in Q_{h}^{k-1}$, and by \eqref{eq:psi-1} it follows that $q- \dive \boldsymbol{\psi} \in Q_h^{k-1,\perp}(D)$. 
    Thus, we may apply the already proved statement to $q- \dive \boldsymbol{\psi}$, as done in the first step. 
	
	To construct $\boldsymbol{\psi}$, we 
    consider $T_i \in \mathcal{T}_h$, $i = 1, \ldots 4$ ordered clockwise in $D$ such that we have
  $D=T_1 \cup T_2 \cup T_3 \cup T_4$  and we let $e_i \coloneqq \partial T_i\cap \partial T_{i+1}$ be the edge shared by $T_i$ and $T_{i+1}$, with indices understood as  modulo $4$. 
  We let $\bn_i$ be the corresponding normal vector of $e_i$ pointing out of $T_i$. 
    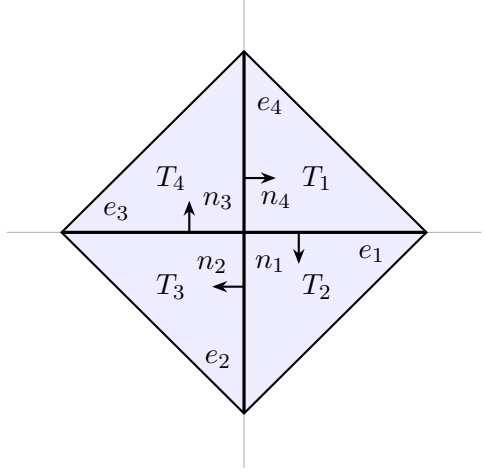
\begin{figure}
        \centering
        \begin{tikzpicture}[scale=1.2, >=Stealth]
    		
    		\coordinate (O) at (0,0);
    		\coordinate (E) at (2,0);
    		\coordinate (N) at (0,2);
    		\coordinate (W) at (-2,0);
    		\coordinate (S) at (0,-2);
    		
    		\draw[gray!60, line width=0.4pt] (-2.6,0)--(2.6,0);
    		\draw[gray!60, line width=0.4pt] (0,-2.6)--(0,2.6);
    		
    		\fill[blue!7]  (O)--(E)--(N)--cycle;   % T1
    		\fill[blue!7]  (O)--(S)--(E)--cycle;   % T2
    		\fill[blue!7]  (O)--(W)--(S)--cycle;   % T3
    		\fill[blue!7]  (O)--(N)--(W)--cycle;   % T4
    		
    		\draw[thick] (N)--(E)--(S)--(W)--cycle;
    		
    		\draw[very thick] (O)--(E);
    		\draw[very thick] (O)--(S);
    		\draw[very thick] (O)--(W);
    		\draw[very thick] (O)--(N);
    		
    		\coordinate (m1) at ($(O)!0.3!(E)$);
    		\coordinate (m2) at ($(O)!0.3!(S)$);
    		\coordinate (m3) at ($(O)!0.3!(W)$);
    		\coordinate (m4) at ($(O)!0.3!(N)$);
    		
    		\coordinate (m10) at ($(O)!0.7!(E)$);
    		\coordinate (m20) at ($(O)!0.7!(S)$);
    		\coordinate (m30) at ($(O)!0.7!(W)$);
    		\coordinate (m40) at ($(O)!0.7!(N)$);
    		
    		\node[below=1pt] at (m10) {$e_1$};
    		\node[left =1pt] at (m20) {$e_2$};
    		\node[above=1pt] at (m30) {$e_3$};
    		\node[right=1pt] at (m40) {$e_4$};
    		
    		\draw[-{Stealth[length=2mm]},thick] (m1) -- ++(0,-0.35) node[left=1pt] {$n_1$}; 
    		\draw[-{Stealth[length=2mm]},thick] (m2) -- ++(-0.35,0) node[above =1pt] {$n_2$}; 
    		\draw[-{Stealth[length=2mm]},thick] (m3) -- ++(0, 0.35) node[right =1pt] {$n_3$}; 
    		\draw[-{Stealth[length=2mm]},thick] (m4) -- ++(0.35,0) node[below =1pt] {$n_4$}; 
    		
    		\node at (0.8, 0.6) {$T_1$};
    		\node at (0.8,-0.6) {$T_2$};
    		\node at (-0.8,-0.6) {$T_3$};
    		\node at (-0.8, 0.6) {$T_4$};	
    	\end{tikzpicture}
        \caption{Example of a domain $D$.}
        \label{fig:placeholder}
    \end{figure}
    We then consider 
	\begin{alignat*}{1}
		\bv_i=  \frac{1}{c_i}b_{e_i} \bn_i,
	\end{alignat*}
	where $b_{e_i}$ denotes the standard quadratic edge bubble function with $\support (b_{e_i} )\subset T_{i} \cup T_{i+1}$,
	and $c_i \coloneqq \int_{e_i} b_{e_i}\ds = |e_i|/6$.
	Then, we see that  $\bv_i$ has support in $T_i \cup T_{i+1}$  and
	\begin{alignat}{1}\label{eq:divTi}
		\int_{T_i} \dive \bv_i\dx= \frac{1}{c_i}\int_{e_i} b_{e_i}\ds=1,
	\end{alignat}
	which also shows that
	\begin{align}\label{eq:divTiplus1}
		\int_{T_{i+1}} \dive \bv_i\dx= -1.  
	\end{align}
	Setting $a_i \coloneqq \int_{T_i} q\dx$, for $i \in \{1,\ldots, 4\}$  we see that $a_1+a_2+a_3+a_4=0$
	because $q\in \mathring{L}^2(\Omega)$ and $\support(q) \subset D$.
	
	We define 
	\begin{equation}\label{def:psi}
		\boldsymbol{\psi}=a_1 \bv_1+(a_1+a_2) \bv_2+ (a_1+a_2+a_3) \bv_3,
	\end{equation}
    and note that $\support(\bpsi) \subset D$. 
	By construction with \eqref{eq:divTi} and \eqref{eq:divTiplus1} we have 
	\begin{alignat*}{1}
		\int_{T_i} \dive \boldsymbol{\psi}\dx= a_i   \qquad  1 \le i \le 3.
	\end{alignat*}
	Moreover,  thanks to $a_1 + a_2 + a_3 + a_4 = 0$ we obtain
	\begin{alignat*}{1}
		\int_{T_4} \dive \boldsymbol{\psi}\dx= -(a_1+a_2+a_3)=a_4.   
	\end{alignat*}
	{This proves~\eqref{eq:psi-1}. 
	To show \eqref{eq:psi-2} let us note that for any $i\in \{1,2,3,4\}$, we have by Hölder's inequality and scaling properties that 
	\begin{align*}
			\abs{a_i} &
			 \leq h_{D}^{2(p-1)/p} \norm{q}_{L^p(D)} \qquad \text{and}\quad
			\norm{\bv_{i}}_{L^\infty(D)} \leq \frac{C}{h_D}. 
	\end{align*}
	Then, by an inverse estimate, the definition of $\bpsi$ in \eqref{def:psi} as well as the estimates above, we find that for any $i\in \{1,2,3,4\}$ we have
	\begin{align*}
		\norm{\nabla \bpsi}_{L^p(T_i)} \leq C h_D^{2/p-1}\norm{\bpsi}_{L^\infty(T_i)} \leq C \norm{q}_{L^p(D)},
	\end{align*}
	which proves \eqref{eq:psi-2}, and hence finishes the proof. 
	}

\section{Proof of Lemma \ref{mainlemmaWithSlip}}\label{app:ProofWithSlip}

The proof of Lemma \ref{mainlemmaWithSlip} relies
on the following extension of \cite[Lem.~6]{GuzmanScott19}
	to the case of slip boundary conditions, where $\tilde Q_{h}^{k-1}$ is the
	discrete pressure space (cf.~\eqref{eqn:tQh}).
    
	%%%%%%%%%%%%%%%%%%%%%%%%%%%
	%%%%%%%%%%%%%%%%%%%%%%%%%%%
	%%%%%%%%%%%%%%%%%%%%%%%%%%%
	\begin{lemma}\label{lem:GZExtension}
		For every $q\in \tilde Q_h^{k-1}$
		and any $z\in \calS_h$, there exists a function $\bv\in \tilde \bV_h^4$ such that
		\begin{subequations}
			\label{eqn:lem6ext}
			\begin{alignat}{2}
				\dive \bv(\sigma) & = 0\qquad &&\text{ for any } \sigma \in \calS_h \text{ with } \sigma\neq z,\\
				\dive \bv|_T(z) & = q|_T(z)\qquad && \text{ for any } T \in \calT_h(z),\\
				\support(\bv)&\subset \Omega_h(z),\\
				\int_K \dive \bv \dx &=0\qquad &&\text{ for any } K\in \calT_h.
			\end{alignat} 
		\end{subequations}
		Moreover,  one has
		\begin{align*}
			\|\nabla \bv\|_{L^p(\Omega_h(z))}\le C\left(\frac{1}{\Theta(z)}+1\right)\|q\|_{L^p(\Omega_h(z))},
		\end{align*}
        with constant $C>0$ only depending on $p, k$ and on the local shape regularity constant.
	\end{lemma}
	\begin{proof}
		If $z\in \calS_h\backslash \calS_h^{2,\partial}$, then the proof 
		of \cite[Lem.~6]{GuzmanScott19} shows
		that there exists $\bv\in \mathring{\bV}_h^4 \subset \tilde \bV_h^4$
		satisfying \eqref{eqn:lem6ext}. Moreover, 
		the proof of \cite[Lem.~6]{GuzmanScott19} shows that $\bv\cdot \bn_e=0$
		on all edges $e$ in $\calT_h$, where $\bn_e$ is a unit normal vector with respect to $e$. 
		
		Therefore it suffices 
		to consider the case $z\in \calS_h^{2,\partial}$.
		We split the proof into three cases, based on the cardinality
		of $\calT_h(z)$.
		
		{\bf Case $\# \calT_h(z) = 2$.}
		In this case we choose $T_1, T_2$ such that  $\calT_h(z) = \{T_1,T_2\}$.
		We reflect these two triangles with 
		respect to the line $\partial\Omega\cap (\partial T_1\cup \partial T_2)$
		to create a set of four triangles $\tilde \calT_h(z) = \{T_1,T_2,T_3,T_4\}$,
		labeled such that $T_i$ and $T_{i+1}$ share a common edge.
		We then see that $z$ is an {\em interior} singular vertex
		with respect to the enlarged triangulation $\calT_h^z \coloneqq \calT_h\cup \{T_3,T_4\}$.
		We extend $q$ to $\calT_h^z$ (still denoted by $q$), by naturally extending $q|_{T_2}$ to $T_3$
		and $q|_{T_1}$ to $T_4$. We then have $q|_{T_1}(z)-q|_{T_2}(z)+q|_{T_3}(z)-q|_{T_4}(z)=0.$
		Consequently, by \cite[Lem.~6]{GuzmanScott19},
		there exists $\tilde \bv\in \mathring{\bV}_h^4$ satisfying \eqref{eqn:lem6ext} with respect
		to $\calT_h^z$ and $\tilde \bv\cdot \bn_e=0$
		on all edges in $\calT_h^z$. Here, $\mathring{\bV}_h^4$ represents the quartic
        Lagrange space with respect to the $\calT_h^z$ with vanishing trace
        on the boundary of $\Omega\cup T_3\cup T_4$.
		Setting $\bv$ to be the restriction of $\tilde \bv$ to $\Omega$,
		we find that $\bv\in \tilde \bV_h^4$ and satisfies \eqref{eqn:lem6ext}
		
		{\bf Case $\# \calT_h(z) = 3$.}
		We choose $\calT_h(z) =\{T_1,T_2,T_3\}$ such that $T_i$ and $T_{i+1}$ share 
		a common edge. We then connect a vertex of $T_1$ and $T_3$ by a line
		to create a fourth triangle $T_4$. Then $z$ is an interior singular vertex
		with respect to the triangulation $\calT^z_h\coloneqq \calT_h\cup \{T_4\}$.
		We can then apply the same argument as in the previous case,
		where we extend $q$ as $q|_{T_4} = q|_{T_1}-q|_{T_2}+q|_{T_3}$.
		
		{\bf Case $\# \calT_h(z) = 1$.}
		The proof of the case $\calT_h(z)$ consisting of one triangle
		is similar to the previous two cases, and is therefore omitted.
	\end{proof}

The proof of Lemma \ref{mainlemmaWithSlip} now
directly follows from the arguments in the proof of Lemma \ref{mainlemma},
but replacing \cite[Lem.~6]{GuzmanScott19} with Lemma \ref{lem:GZExtension}
in Step 1 in the proof, and replacing $Q_h^{k-1}$ and $\mathring{\bV}_h^k$
with $\tilde Q_h^{k-1}$ and $\tilde \bV_h^k$, respectively.

\end{document}